\documentclass[12pt,a4paper,oneside,reqno]{amsart}

\usepackage{graphicx}
\usepackage{amsmath} % Lots of maths functionality
\usepackage{amssymb} % Maths symbols
\usepackage{mathrsfs} % Maths scripted font.
\usepackage{amsthm} % Maths environments: \begin{proof}, etc.
\usepackage[english]{babel} % Language and hyphenation.
\usepackage[T1]{fontenc} % For font encoding, to allow accents, copy & paste, inequality signs, etc. to all work nicely.
\usepackage[utf8]{inputenc} % To be loaded after fontenc, also for encoding.
\usepackage{mathtools} % Uses amsmath, fixes quirks and adds functionality.
\usepackage[pdftex,  colorlinks=false]{hyperref} % Makes references and citations into links
\usepackage{setspace} % Allows \onehalfspacing etc. for changing gaps between lines
\usepackage{tikz-cd} % Commutative diagrams
\usepackage[capitalize]{cleveref} % use \Cref{} for instead of X~\ref{} 

\usepackage{pgf,tikz,pgfplots}
\usetikzlibrary{arrows,quotes,calc}
\tikzstyle{blackNode}=[fill=black, draw=black, shape=circle]

\numberwithin{equation}{section}

\makeatletter
\@namedef{subjclassname@1991}{Mathematical subject classification 1991}
\@namedef{subjclassname@2000}{Mathematical subject classification 2000}
\@namedef{subjclassname@2010}{Mathematical subject classification 2010}
\@namedef{subjclassname@2020}{Mathematical subject classification 2020}
\makeatother

\newtheorem{theorem}{Theorem}[section]
\newtheorem{lemma}[theorem]{Lemma}
\newtheorem{corollary}[theorem]{Corollary}
\newtheorem{proposition}[theorem]{Proposition}

\newtheorem{thmx}{Theorem}

\newtheorem{meta-question}[mquex]{Meta-question}

\theoremstyle{definition}
\newtheorem{definition}[theorem]{Definition}

\newtheorem{example}[theorem]{Example}

\newtheorem*{examples*}{Examples}

\newtheoremstyle{TheoremNum}
{\topsep}{\topsep} %%% space between body and thm
{\itshape} %%% Thm body font
{-0.25cm} %%% Indent amount (empty = no indent)
{\bfseries} %%% Thm head font
{.} %%% Punctuation after thm head
{ }  %%% Space after thm head
{\thmname{#1}\thmnote{ \bfseries #3}}%%% Thm head spec
\theoremstyle{TheoremNum}

\DeclareMathOperator{\aut}{\mathrm{Aut}}
\DeclareMathOperator{\out}{\mathrm{Out}}

\DeclareMathOperator{\im}{\mathrm{im}}
\DeclareMathOperator{\id}{\mathrm{Id}}

\DeclareMathOperator{\ad}{ad}
\DeclareMathOperator{\diag}{diag}
\DeclareMathOperator{\mcg}{MCG}

\newcommand{\GL}{\mathrm{GL}}

\newcommand{\SL}{\mathrm{SL}}

\newcommand{\gen}[1]{\langle #1 \rangle}
\newcommand{\pres}[2]{\langle #1\ |\ #2 \rangle}

\DeclareMathOperator{\lcm}{\mathrm{lcm}}

\newcommand{\N}{\mathbb{N}}
\newcommand{\Z}{\mathbb{Z}}

\newcommand{\cD}{\mathcal{D}}

\newcommand{\sub}{\subset}

\usepackage[foot]{amsaddr}
\title[Lack of profinite rigidity among `by-free' extensions]{Lack of profinite rigidity among extensions with free quotient}
\author{Pawe\l\ Piwek}
\date{November 2023}
\address{Mathematical Institute, Andrew Wiles Building, Observatory Quarter, University of Oxford, Oxford OX2 6GG, UK}
\email{pawel.piwek@maths.ox.ac.uk}

\begin{document}
	
\newpage
	
\maketitle

\begin{abstract}
	We present a construction that yields
	infinite families of non-isomorphic
	semidirect products $N \rtimes F_m$
	sharing a specified profinite completion.
	Within each family, $m \ge 2$ is constant
	and $N$ is a fixed group.
	For $m=2$ we can take
	$N$ to be free of rank $\ge 10$,
	free abelian of rank $\ge 12$,
	or a surface group of genus $\ge 5$.
\end{abstract}
	
	\section{Introduction}\label{sec:intro}
	
	\subsection{Context and related results}\label{ssec:context}
	
	The study of profinite rigidity properties
	of residually finite groups
	is a highly active area of group theory.
	Before stating the original results of this paper,
	we describe two particular
	interconnected directions in this research.
	The first one concerns the search for `large' families
	of non-isomorphic residually finite groups
	whose profinite completions are isomorphic.
	The second concerns the challenge of determining when
	a given family of \mbox{$N$-by-$Q$} extensions
	can be distinguished one from another
	by their profinite completions.
	
	\subsubsection{Profinitely isomorphic families}
	
	Simplifying earlier work \cite{serre1964exemples} of Serre,
	in \cite{baumslag1974residually}
	Baumslag gave an example of
	two non-isomorphic semidirect products ${\left(\Z/25\right) \rtimes \Z}$
	which have the same finite quotients.
	Many other families
	of non-isomorphic residually finite groups
	sharing profinite completions
	were discovered at about the same time,
	but these early examples were polycyclic
	and the families were finite.
	Indeed, virtually-polycyclic groups  
	were shown in \cite{grunewald1979finiteness}
	to have finite profinite genus.
	
	Pickel gave in \cite{pickel1974metabelian}
	an infinite family
	of non-isomorphic finitely presented metabelian groups
	which all shared profinite completions.
	His construction involved
	finding $R$-modules $M_i$,
	which have the same finite quotients (as $R$-modules),
	and then forming $G_i := M_i \rtimes A$
	for a subgroup $A < R^\times$.
	
	An important direction of research was
	finding so called \emph{Grothendieck pairs}:
	pairs $H \xhookrightarrow{\iota} G$
	of residually-finite groups
	such that $\widehat{H} \xhookrightarrow{\widehat{\iota}} \widehat{G}$
	is an isomorphism,
	but $\iota$ itself isn't.
	In 1990 Platonov and Tavgen---%
	building on the earlier work
	\cite{baumslag1984subgroups}
	of Baumslag and Roseblade---%
	showed in
	\cite{platonov1990grothendieck}
	that a product of two
	nonabelian free groups
	contains uncountably many
	pairwise non-isomorphic subgroups
	inducing Grothendieck pairs;
	furthermore, infinitely many
	of them are finitely generated,
	but none is finitely presented.
	Grothendieck pairs of finitely presented groups
	were later constructed in \cite{bridson2004grothendieck}.
	Subsequently,
	an infinite family of Grothendieck pairs
	$H_i \hookrightarrow \Gamma \times \Gamma$
	with $H_i$ finitely presented
	and $\Gamma$ hyperbolic was later constructed
	in \cite{bridson2016strong}.
	For a modern discussion           
	see Bridson's essay \cite{bridson2022profinite}.
	
	In 2014 Nekrashevych gave in \cite{nekrashevych2014uncountable} 
	a family of finitely generated residually finite
	branch groups which all share profinite completions,
	but this family also contains uncountably many word-growth types,
	so in particular, uncountably many isomorphism types.
	Recently, \cite{kionke2021amenability}
	gave another family $\{G_i\}_{i\in I}$
	of pairwise non-isomorphic branch groups, all containing $F_2$
	and an amenable branch group $A$
	such that $\iota_i : A \hookrightarrow G_i$
	give Grothendieck pairs.
	
	\subsubsection{Profinite rigidity properties and group extensions}
	
	Many of the results concerning profinite rigidity
	(or the lack thereof)
	relate closely to the question
	of whether different group extensions
	with a fixed kernel $N$
	and a fixed quotient $Q$
	are profinitely isomorphic.
	In fact, the early examples of Baumslag
	\cite{baumslag1974residually}
	and of Pickel \cite{pickel1974metabelian}
	are precisely of this `flavour'.
	
	A survey covering many such results
	in the context of virtually free abelian
	and virtually free groups is the 2011 paper \cite{grunewald2011genus}
	of Grunewald and Zalesskii.
	In particular, it shows that the genus
	of any virtually-(finitely generated free) group is finite within this class
	and in some cases determines the genus explicitly,
	thereby giving interesting examples of non-isomorphic
	but profinitely isomorphic virtually free groups
	and also conditions which force the genus to be size one.
	The survey also provides a key ingredient of the present paper
	-- \cref{semidirect_products_profinitely_isomorphic}.
	
	Many of the efforts to distinguish $3$-manifold groups
	by their profinite completions
	also involve the study of extensions.
	The families of SOL manifolds given in \cite{funar2013torus} in 2013 by Funar
	have fundamental groups of the form $\Z^2 \rtimes \Z$
	and the fact that there are non-isomorphic pairs among them
	whose profinite completions are isomorphic
	comes from the existence of non-conjugate matrices
	in $\SL_2(\Z)$ that are conjugate
	in $\SL_2(\widehat \Z)$, which was shown
	in \cite{stebe1972conjugacy} in 1972 by Stebe.
	In 2017 Bridson, Reid and Wilton showed
	in \cite{bridson2017profinite}
	that non-isomorphic extensions $F_2 \rtimes \Z$
	are profinitely distinguished from each other.
	In the same year Wilkes showed in \cite{wilkes2017profinite} that
	the fundamental groups of Seifert Fibred Spaces---%
	which are extensions of $\Z$ by $2$-orbifold groups---%
	are profinitely distinguished from each other
	with the exception of the examples given by Hempel in 2014 in \cite{hempel2014some}.
	Later, the present author extended this classification
	to central extensions
	with kernel $\Z^n$ and quotient being an $2$-orbifold group
	in \cite{piwek2023profinite}.
	
	It should be noted that there are
	at least two separate phenomena causing lack of rigidity
	in the described examples.
	One is that the `induced actions' on the kernel
	of the extension can be non-conjugate in the discrete setting,
	but become conjugate after passing to the profinite completions.
	These include the families
	$(\Z/25) \rtimes \Z$, $\Z^2 \rtimes \Z$
	and many of the examples of \cite{grunewald2011genus}.
	The other phenomenon is when the the actions are the same
	(e.g. the extensions might both be central),
	but the extensions are distinct as
	the cocycles defined by them
	lie in distinct orbits of the action
	of the quotient group
	on the relevant second cohomology group,
	but they lie in the same orbit
	when we pass to profinite cohomology.
	This last part explains the examples of Hempel
	and of \cite{piwek2023profinite}.
	
	The method for producing profinitely isomorphic families
	which we give in this article
	is underpinned by the first phenomenon.
	
	\subsection{Original results}\label{ssec:original}
	
	The present paper studies
	families of semidirect products $N \rtimes F_m$
	of a fixed group $N$
	with a non-abelian free group $F_m$
	which are pairwise non-isomorphic,
	but whose profinite completions are isomorphic.
	
	\cref{mainthmintro:general}
	provides a general `recipe' for constructing such examples
	given a non-free subgroup ${T < \aut(N)}$
	with multiple surjections
	${F_m \twoheadrightarrow \overline T < \out(N)}$
	which are non-equivalent under the action of $\aut(F_m) \times \out(N)$.
	An example of when this construction can be applied
	is when $N$ is a free group,
	a surface group or a free abelian group.
	\cref{mainthmintro:infinite_genus_families}
	makes this explicit.
	
	\begin{thmx}\label{mainthmintro:general}
		Let $N$ be a
		finitely generated,
		residually finite group
		such that
		the centralisers $C_N(M)$
		of non-trivial normal subgroups $M \triangleleft N$
		don't contain a non-abelian free group.
		
		Let $T < \aut(N)$ be a subgroup which isn't free.
		Let $m \ge 2$ and $\{\varphi_i\colon F_m \twoheadrightarrow T\}$
		be a family of surjective homomorphisms such that,
		setting $\overline{\varphi_i}$ to be the composition of $\varphi_i$
		with ${\nu\colon \aut(N) \twoheadrightarrow \out(N)}$
		and $\overline T = \nu(T)$,
		we have
		\begin{equation*}
			\delta \circ \overline{\varphi_i} = \overline{\varphi_j} \circ \varepsilon\
			\text{for some $\delta \in \aut(\overline T)$,
				$\varepsilon \in \aut(F_m)$}\
			\iff\ i = j.
		\end{equation*}
		
		Then the groups $G_i := N \rtimes_{\varphi_i} F_m$
		are residually-finite
		and pairwise non-isomorphic,
		while having isomorphic profinite completions.
	\end{thmx}

	\begin{thmx}\label{mainthmintro:infinite_genus_families}
		Let $N$ be one of the following.
		\begin{itemize}
			\item Free group $F_n$ with $n \ge 10$.
			\item Fundamental group $\Sigma_n$
			of a closed orientable surface with genus $n \ge 5$, or
			\item Free abelian group $\Z^n$ with $n \ge 12$.
		\end{itemize}
		Then there exists an infinite family
		of non-isomorphic groups
		$G_i := N \rtimes_{\varphi_i} F_2$
		having isomorphic profinite completions.
		These groups are of type $F$.
	\end{thmx}
	
	\subsection{How this paper is structured}\label{ssec:structure}
	
	The remainder of this article is structured as follows.
	
	\cref{sec:profinite_gps} introduces the results
	related to profinite groups used later in the article.
	Similarly, \cref{sec:group_extensions} introduces
	the background on group extensions, specifically
	distinguishing their isomorphism types in \cref{ssec:extensions_non-isomorphic}
	and giving a method for proving they are isomorphic
	in \cref{ssec:extensions_isomorphic}.
	
	\cref{sec:detecting_isomorphisms} introduces
	a criterion for `characterising'
	the subgroups $N \rtimes 1$
	in semidirect products $N \rtimes F_m$,
	using centralisers of normal subgroups.
	
	\cref{sec:T-systems} discuses the topic of
	Nielsen-equivalence and T-equivalence,
	citing in \cref{ssec:torus_link_groups} results
	on torus links having infinitely
	many non-equivalent T-systems.
	\cref{ssec:finding_torus_link_gps} explains then
	how to find the torus links as subgroups of
	outer automorphism groups of free groups, surface groups
	and free abelian groups.
	
	Finally, \cref{sec:final_results} states and proves
	the main results.
	
	\subsection*{Acknowledgements}
	
	The author is thankful to his PhD supervisor,
	Martin Bridson, for numerous helpful conversations and suggestions.
	
	This work was supported by the Mathematical Institute Scholarship
	of University of Oxford.
	
	\section{Auxiliary results on profinite groups}\label{sec:profinite_gps}
	
	\Cref{semidirect_after_completion} shows
	that (in the case of finitely generated kernel)
	a profinite completion of a semidirect product
	is the semidirect product of the profinite completions.
	
	\begin{proposition}\label{semidirect_after_completion}
		Let $N$	and $Q$ be any groups.
		
		Then $\widehat{N \rtimes Q} \cong \overline{N} \rtimes \widehat{Q}$,
		where $\overline{N}$ is the closure of the image of $N$ in $\widehat{N \rtimes Q}$.
		
		Furthermore, if $N$ is finitely generated, then $\overline{N} \cong \widehat{N}$,
		and if $N$ and $Q$ are residually finite, then so is $N \rtimes Q$.
	\end{proposition}
	
	\begin{proof}
		Let
		\begin{tikzcd}[column sep=1.5em]
			1 \arrow{r} & N \arrow{r}{\iota} & N\rtimes Q \arrow{r}{\pi} & Q \arrow{r} \arrow[bend right=45, swap]{l}{s} & 1
		\end{tikzcd}
		be a short exact sequence with $\pi\circ s = \id_Q$.
		Profinite completion turns this into another short exact sequence
		\begin{equation*}
			\begin{tikzcd}[column sep=2em]
				1 \arrow{r} & \overline{N} \arrow{r}{\iota'} &
				\widehat{N\rtimes Q} \arrow{r}{\widehat{\pi}} &
				\widehat{Q} \arrow{r} \arrow[bend right=45, swap]{l}{\widehat{s}} & 1
			\end{tikzcd}
		\end{equation*}
		with $\widehat{\pi}\circ \widehat{s} = \widehat{\id_Q} = \id_{\widehat{Q}}$
		-- for details see \cite[Proposition 3.2.5.]{ribes2000profinite}.
		This is a split short exact sequence,
		so $\widehat{N \rtimes Q} \cong \overline{N} \rtimes \widehat{Q}$.
		
		Now, if $N$ is finitely generated, then it has a cofinal sequence $N_1>N_2>\ldots$
		of finite-index characteristic subgroups.
		We have a sequence of subgroups $N_i \rtimes Q$
		of finite index in $N \rtimes Q$,
		such that for any $M < N$ of finite index,
		$(N_i \rtimes Q) \cap N < M$ for some $i$.
		Thus, the topology induced by $\widehat{N \rtimes Q}$ on the image of $N$
		is the full profinite topology and so $\overline{N} \cong \widehat{N}$.
		
		Finally, if $N$ and $Q$ are residually-finite,
		then they canonically embed in their profinite completions
		and so also $N \rtimes Q$ embeds into $\widehat{N \rtimes Q} \cong \widehat{N} \rtimes \widehat{Q}$
		via the natural map, so $N \rtimes Q$ is residually-finite.
	\end{proof}
	
	Gaschutz' Lemma for profinite groups
	will be crucial for us
	ensuring in \cref{semidirect_products_profinitely_isomorphic}
	that there are isomorphisms between the profinite completions
	of the groups in the families we construct.
	\Cref{gaschutz} and \cref{Nielsen_profinite}
	are respectively \cite[Lemma 4.2.]{jarden1975elementary}
	and \cite[Lemma 4.3.]{jarden1975elementary}.
	
	\begin{lemma}[Gaschutz' Lemma]\label{gaschutz}
		Let $\Gamma$ and $\Delta$ be profinite groups
		such that
		${\Gamma = \overline{\gen{\gamma_1, \ldots, \gamma_d}}}$
		for some $\gamma_1, \ldots, \gamma_d \in \Gamma$,
		and let $p: \Gamma \twoheadrightarrow \Delta$
		be a surjective homomorphism of profinite groups.
		
		Then given any $(\delta_1, \ldots, \delta_d) \in \Delta^d$
		such that $\overline{\gen{\delta_1, \ldots, \delta_d}} = \Delta$,
		there exists $(\gamma_1', \ldots, \gamma_d') \in \Gamma^d$
		such that $p(\gamma_i') = \delta_i$ for $i = 1, \ldots, d$
		and such that $\overline{\gen{\gamma_1', \ldots, \gamma_d'}} = \Gamma$.
	\end{lemma}
	
	\begin{corollary}\label{Nielsen_profinite}
		Let $\Delta$ be a profinite group,
		and $p_i\colon \widehat{F_m} \twoheadrightarrow \Delta$
		be a surjective homomorphism for $i = 1, 2$.
		
		Then there exists an automorphism
		$\varepsilon \colon \widehat{F_m} \to \widehat{F_m}$
		such that $p_1 = p_2 \circ \varepsilon$.
	\end{corollary}

	\section{Group extensions}\label{sec:group_extensions}
	
	The topic of group extensions can be confusing due to unclear definitions
	and a number of different notions of being 'the same'.
	We try to avoid this confusion here.
	
	\begin{definition}
		An \emph{$N$-by-$Q$ extension} is a group $G$
		together with a short exact sequence
		\begin{equation*}
			\begin{tikzcd}
				1 \arrow{r} & N \arrow{r}{\iota} & G \arrow{r}{\pi} & Q \arrow{r} & 1.
			\end{tikzcd}
		\end{equation*}
	\end{definition}
	
	\begin{definition}\label{def:the_same_extensions}
		For $i=1, 2$ the extensions
		$\big(1 \to N \xrightarrow{\iota_i} G_i \xrightarrow{\pi_i} Q \to 1\big)$
		\begin{itemize}
			\item are called \emph{isomorphic as groups}
			if there exists an isomorphism $\beta: G_1 \to G_2$;
			\item are called \emph{similar}
			if there exist isomorphisms
			$\alpha: N \to N$, ${\beta: G_1 \to G_2}$ and $\gamma: Q \to Q$
			such that the diagram~\ref{extensions:similar} commutes;
			\begin{equation}\label{extensions:similar}
				\begin{tikzcd}
					1 \arrow{r} &
					N \arrow{r}{\iota_1} \arrow[]{d}{\alpha} &
					G_1 \arrow{r}{\pi_1} \arrow{d}{\beta} &
					Q \arrow{r} \arrow[]{d}{\gamma} &
					1\\
					1 \arrow{r} &
					N \arrow{r}{\iota_2} &
					G_2 \arrow{r}{\pi_2} &
					Q \arrow{r} &
					1
				\end{tikzcd}
			\end{equation}
			\item are called \emph{equivalent}
			if there exists an isomorphism $\beta: G_1 \to G_2$
			such that the diagram~\ref{extensions:equivalent} commutes.
			\begin{equation}\label{extensions:equivalent}
				\begin{tikzcd}
					1 \arrow{r} &
					N \arrow{r}{\iota_1} \arrow[equals]{d} &
					G_1 \arrow{r}{\pi_1} \arrow{d}{\beta} &
					Q \arrow{r} \arrow[equals]{d} &
					1\\
					1 \arrow{r} &
					N \arrow{r}{\iota_2} &
					G_2 \arrow{r}{\pi_2} &
					Q \arrow{r} &
					1
				\end{tikzcd}
			\end{equation}
		\end{itemize}
	\end{definition}
	
	\begin{definition}
		Given an $N$-by-$Q$ extension
		$\big(1 \to N \xrightarrow{\iota} G \xrightarrow{\pi} Q \to 1\big)$
		the conjugation action of $G$ on itself induces
		an action by automorphisms on $N$;
		this induces a homomorphism $\varphi: Q \to \out(N)$
		called \emph{the induced homomorphism} of the extension.
	\end{definition}
	
	In fact equivalent extensions give the same induced homomorphism.
	For similar extensions it can change though,
	which is described in \Cref{induced_vs_similar}.
	
	\begin{proposition}\label{induced_vs_similar}
		For $i = 1, 2$, let
		$\big(1 \to N \xrightarrow{\iota_i} G_i \xrightarrow{\pi_i} Q \to 1\big)$
		be similar extensions,
		with $\alpha: N \to N$ and $\gamma: Q \to Q$
		as in diagram~\ref{extensions:similar}.
		Then, letting $\overline\alpha$ denote the image of $\alpha$ in $\out(N)$,
		the induced homomorphisms $\varphi_i: Q \to \out(N)$ satisfy
		\begin{equation*}
			\forall q \in Q:\quad
			\overline\alpha\cdot \varphi_1(q) \cdot \overline\alpha^{-1}
			= (\varphi_2 \circ \gamma)(q).
		\end{equation*}
		In particular,
		the induced homomorphisms of equivalent extensions
		are equal.
	\end{proposition}
	
	\begin{proof}
		Let's denote by $\tilde{\varphi}_i: G_i \to \aut(N)$ the conjugation actions.
		Then
		\begin{equation*}
			\forall g\in G_1:\quad \alpha\cdot\tilde{\varphi}_1(g)\cdot \alpha^{-1}
			= (\tilde{\varphi}_2\circ\gamma)(g).
		\end{equation*}
		Applying $\aut(N) \twoheadrightarrow \out(N)$
		we get $\overline{\alpha}\cdot \varphi_1(q) \cdot \overline{\alpha}^{-1}
		= (\varphi_2 \circ \gamma)(q)$.
	\end{proof}
	
	\subsection{Non-isomorphic extensions}\label{ssec:extensions_non-isomorphic}
	
	The spirit of this article is very much that of
	distinguishing groups up to isomorphism
	using the information given by writing them
	in the form of group extensions.
	However, for this method to be applicable,
	we need to know that isomorphisms between these groups
	`upgrade' to extension similarities.
	
	\begin{definition}
		A class $\cD$ of $N$-by-$Q$ extensions
		\emph{detects isomorphisms}
		if given extensions $\big(
		\begin{tikzcd}[column sep=1.5em]
			1 \arrow{r}	& N \arrow{r}{\iota_i} & G_i \arrow{r}{\pi_i} & Q \arrow{r} & 1
		\end{tikzcd}\big) \in \cD$
		for ${i = 1, 2}$
		and any isomorphism ${\beta: G_1 \to G_2}$,
		we have $\beta(\iota_1(N)) = \iota_2(N)$.
		In particular,
		any group isomorphism of these extensions
		induces a similarity of extensions.
	\end{definition}
	
	One should note that if $\cD$ detects isomorphisms
	then the images $\iota_i(N)$
	are characteristic in $G_i$,
	but the reverse implication doesn't hold%
	---see \cref{detecting_stronger_than_characteristic}.
	
	\begin{example}\label{detecting_stronger_than_characteristic}
		Let $G = \left(\Z/4\right) \rtimes \left(\Z/8\right)$ be
		the non-trivial semidirect product of
		$\Z/4 = \gen{x}$ with $\Z/8 = \gen{y}$.
		Let $N_1 = \gen{xy^2, y^4}$
		and $N_2 = \gen{x^3y^2, y^4}$.
		
		Then both subgroups $N_i$ are isomorphic to
		$\left(\Z/4\right) \times \left(\Z/2\right)$,
		both are characteristic in $G$,
		and $G/N_1 \cong G/N_2 \cong \Z/4$,
		but there is no automorphism $\beta \in \aut(G)$
		such that $\beta(N_1) = N_2$.
	\end{example}
	
	The usefulness of the concept of detecting isomorphisms
	becomes clear in \cref{different_maps_to_out},
	which will be our tool for distinguishing
	isomorphism classes of extensions.
	
	\begin{corollary}\label{different_maps_to_out}
		Let $N, Q$ be groups
		and $\cD$ a class of $N$-by-$Q$ extensions,
		which detects isomorphisms.
		
		Let's assume that for any
		$\big(
		\begin{tikzcd}[column sep=1.5em]
			1 \arrow{r}	& N \arrow{r}{\iota_i} & G_i \arrow{r}{\pi_i} & Q \arrow{r} & 1
		\end{tikzcd}\big) \in \cD$,
		the homomorphisms $\varphi_i: Q \to \out(N)$
		induced by $G_i \to \aut(N)$ satisfy
		\begin{equation*}
			\forall\ \overline\alpha \in \out(N),\ \gamma \in \aut(Q):\
			\overline\alpha\varphi_{i}\overline{\alpha}^{-1} \ne \varphi_{j} \circ \gamma\
			\text{as maps $Q \to \out(N)$}.
		\end{equation*}
		
		Then the groups $G_i$ are pairwise non-isomorphic.
	\end{corollary}
	
	\subsection{Isomorphic extensions}\label{ssec:extensions_isomorphic}
	
	In this section we provide in \cref{isomorphic_semidirect_products}
	a criterion for establishing similarity of two group extensions,
	and thus---isomorphism of their groups.
	This criterion gives a simple sufficient condition
	for the isomorphism of profinite completions.
	
	The method, and much of this section
	has been inspired by \cite{grunewald2011genus},
	which studied in detail the profinite rigidity
	of several interesting classes of group extensions.
	A particularly relevant result therein is Proposition 2.8.
	
	\begin{proposition}\label{isomorphic_semidirect_products}
		Let $G_i = N \rtimes_{\varphi_i} Q$ for $i = 1, 2$.
		Suppose that ${\varphi_i\colon Q \to \aut(N)}$
		are such that there exist $\alpha \in \aut(N)$
		and $\gamma \in \aut(Q)$
		such that $\alpha\varphi_1\alpha^{-1} = \varphi_2 \circ \gamma$.
		
		Then $G_1 \cong G_2$ via
		$\beta: (n, q) \mapsto \left(\alpha(n), \gamma(q)\right)$.
	\end{proposition}
	
	\begin{proof}
		We only need to check
		that the map $\beta$ respects group operations.
		\begin{align*}
			\beta\big((n_1, q_1) \cdot (n_2, q_2)\big) &
			= \beta\big((n_1\varphi_1(q_1)(n_2), q_1q_2)\big)
			\\ &
			= \big(\alpha(n_1)\cdot
			(\alpha\circ \varphi_1)(q_1)(n_2),\ \gamma(q_1q_2)\big)
			\\
			\beta\big((n_1, q_1)\big) \cdot
			\beta\big((n_2, q_2)\big) &
			= \big(\alpha(n_1), \gamma(q_1)\big) \cdot
			\big(\alpha(n_2), \gamma(q_2)\big)
			\\ &
			= \big(\alpha(n_1)\cdot \varphi_2(\gamma(q_1))(\alpha(n_2)),\ \gamma(q_1)\gamma(q_2)\big)
		\end{align*}
		but 
		\begin{align*}
			(\alpha \circ \varphi_1)(q_1)(n_2) &=
			(\alpha\varphi_1\alpha^{-1})(q_1)(\alpha(n_2)) \\ & = 
			(\varphi_2 \circ \gamma)(q_1)(\alpha(n_2)) \\ & =
			\varphi_2(\gamma(q_1))(\alpha(n_2)),
		\end{align*}
		so $\phi$ is indeed a group isomorphism.
	\end{proof}
	
	This is true also for profinite groups $N$ and $Q$;
	the situation is further simplified
	if $Q = \widehat{F_m}$.
	
	\begin{corollary}\label{semidirect_products_profinitely_isomorphic}
		Let $N$ be finitely generated
		and ${\varphi_i: F_m \to \aut(N)}$ for ${i=1, 2}$
		be homomorphisms
		with $\im \varphi_1 = \im \varphi_2$.
		Let $G_i = N \rtimes_{\varphi_i} F_m$.
		
		Then $\widehat{G_1} \cong \widehat{G_2}$.
	\end{corollary}
	
	\begin{proof}
		By \cref{semidirect_after_completion},
		the completions $\widehat{G_i}$ are
		semidirect products
		$\widehat{N} \rtimes_{\varphi_i'} \widehat{F_m}$
		where $\varphi_i'$ is the induced map
		\begin{equation*}
			\begin{tikzcd}
				\widehat{F_m} \arrow{r}{\widehat\varphi_i} &
				\widehat{\aut(N)} \arrow{r} &
				\aut(\widehat N)
			\end{tikzcd}.
		\end{equation*}
		Since $\im \varphi_1 = \im \varphi_2$,
		also
		$\im \widehat{\varphi}_1 = \im \widehat{\varphi}_2$
		and thus $\im \varphi_1' = \im \varphi_2'$.
		
		Now, since the images of $\varphi_i'$
		are profinite groups,
		by \cref{Nielsen_profinite},
		we can find $\varepsilon \in \aut(\widehat{F_m})$
		such that
		$\varphi_1' = \varphi_2' \circ \varepsilon$.
		Then by \cref{isomorphic_semidirect_products}
		we get that $\widehat G_1 \cong \widehat G_2$.
	\end{proof}
	
	%\todo{Do we actually need \cref{centreless_iso}? What are we going to use in the proof of \cref{mainthm:general}?}
	%\begin{proposition}\label{centreless_iso}
	%	Let $N$ and $Q$ be groups and $\mathcal{Z}(N) = 1$.
	%	
	%	Then any extension
	%	$\big(
	%	\begin{tikzcd}[column sep=1.5em]
		%		1 \arrow{r}
		%			& N \arrow{r}{\iota}
		%				& G \arrow{r}{\pi}
		%					& Q \arrow{r}
		%						& 1
		%	\end{tikzcd}\big)$
	%	inducing the homomorphism ${\varphi\colon Q \to \out(N)}$
	%	is equivalent to
	%	\begin{equation*}
		%		\big(
		%		\begin{tikzcd}[column sep=1.5em]
			%			1 \arrow{r}
			%			& N \arrow{r}{\iota'}
			%			& P \arrow{r}{\pi'}
			%			& Q \arrow{r}
			%			& 1
			%		\end{tikzcd}\big)
		%	\end{equation*}
	%	where $P$ is the pullback
	%	$\{(f, q) \in \aut(N) \times Q\ |\ \overline f = \varphi(q)\}$
	%	with maps $\iota' \colon N \to P$ and $\pi' \colon P \to Q$
	%	defined as
	%	\begin{equation*}
		%		\iota'(n) := (\ad_{n}, 1), \quad
		%		\pi'\left((f, q)\right) := q.
		%	\end{equation*}
	%
	%	Furthermore, two such extensions
	%	$\big(
	%	\begin{tikzcd}[column sep=1.5em]
		%		1 \arrow{r}
		%		& N \arrow{r}{\iota_i}
		%		& G_i \arrow{r}{\pi_i}
		%		& Q \arrow{r}
		%		& 1
		%	\end{tikzcd}\big)$
	%	are similar if and only if
	%	$^{\overline \alpha} \varphi_1 = \varphi_2 \circ \beta$
	%	for some $\overline\alpha \in \out(N)$ and $\beta \in \aut(Q)$.
	%\end{proposition}

	\section{Detecting isomorphisms of semidirect products with free quotient}
	\label{sec:detecting_isomorphisms}
	
	In proving \cref{mainthm:general}
	we will need to show that group isomorphisms
	$G_i \to G_j$ of the considered $N$-by-$Q$ extensions
	are promoted to extension similarities
	-- i.e. they come from an automorphism on the kernel and on the quotient.
	
	In some situations (such as when $N \cong \Z^n$)
	a simple criterion of \cref{easy}
	inspired by \cite[Remark 3.17]{grunewald2011genus}
	is enough.
	However, since we want to treat the cases of
	$Q$ being free or a surface group,
	a more complicated criterion of
	\Cref{characteristic} will be necessary.
	
	\begin{proposition}\label{easy}
		Let $N$ be a group
		having no non-abelian free quotients
		and $m \ge 2$.
		Then the subgroup $N = N \rtimes 1$
		of $G = N\rtimes_\varphi F_m$
		is the unique normal subgroup of $G$
		such that $G / N \cong F_m$.
	\end{proposition}
	
	\begin{proof}
		Let $M \triangleleft G$ be
		another such normal subgroup.
		Then $\overline N = NM/M$
		is a normal subgroup of $F_m$,
		so it is either trivial,
		or non-abelian free.
		
		Since $N$ has no non-abelian free quotients,
		$\overline N = 1$
		and so $N < M$,
		but this gives a surjection
		\begin{equation*}
			F_m \cong G/N \twoheadrightarrow G/M \cong F_m
		\end{equation*}
		which must be an isomorphism
		by the Hopf property of
		finitely generated free groups.
	\end{proof}
	
	\begin{theorem}\label{characteristic}
		Let $N$ be a group in which
		the centralisers of non-trivial normal subgroups
		don't contain a non-abelian free group.
		
		Let $m \ge 2$,
		let $F_m$ be a finitely generated free group
		and let $T$ be a subgroup of $\aut(F)$ which isn't free.
		Suppose that ${\varphi: F_m \twoheadrightarrow T}$ is a surjection.
		
		Then the subgroup
		$N=N \rtimes 1$ of $G = N \rtimes_\varphi F_m$
		is characterised as the only subgroup $M$ of $G$
		satisfying the following conditions.
		\begin{itemize}
			\item $M$ is isomorphic to $N$,
			\item $M$ is normal in $G$,
			\item $G/M \cong F_m$,
			\item $C_G(M)$ contains a non-abelian free group.		
		\end{itemize}
		In particular, $N$ is characteristic in $G$.
		
		Furthermore, given any isomorphism $\beta\colon G_1 \to G_2$
		of two such groups $G_i := N \rtimes_{\varphi_i} F_m$,
		we have $\beta(N) = N$ and there exist
		$\alpha \in \aut(N)$ and $\gamma \in \aut(F_m)$ such that
		$\alpha \varphi_1 \alpha^{-1} = \varphi_2 \circ \gamma$.
	\end{theorem}

	\subsection{Notation and lemmas}\label{ssec:detecting_isos_notation}
	
	For the sake of clarity
	we will fix some notation for the proof of \cref{characteristic}.
	\begin{align*}
		N & := \text{a group st. for non-trivial $O \triangleleft N$,
			$F_2 \not < C_N(O)$}, \\
		G & := N \rtimes_\varphi F_m, \\
		N & := N \rtimes 1 \\
		H & := 1 \rtimes F_m \\
		K & := 1 \rtimes (\ker \varphi), \\
		T & := H / K, \\
		M & := \text{a subgroup of $G$ satisfying the four conditions}, \\
		\overline G & := G/M, \\
		\overline N, \overline H, \overline K
		& := \text{images of $N, H, K$ in $G \twoheadrightarrow \overline G$}, \\
		L & := \ker(H \twoheadrightarrow \overline H) = H \cap M, \\
		D & := \gen{\gen{n \cdot \varphi(l)(n^{-1})\ |\ l \in L, n \in N}}_N.
	\end{align*}
	
	We will be proving that $M = N$.
	
	\begin{lemma}\label{normal_closure_in_semidirect}
		Let $G = N \rtimes_\varphi H$ be a semidirect product
		and $L \triangleleft H$.
		
		Then $$\gen{\gen{L}}_G
		= \gen{\gen{n \cdot \varphi(l)(n^{-1})\ |\ l \in L, n \in N}}_N \cdot L.$$
	\end{lemma}
	
	\begin{proof}
		Take $(n, 1) \in N$ and $(1, l) \in L$.
		Then
		$$(n, 1) (1, l)(n^{-1}, 1)
		= \left(n \cdot \varphi(l)(n^{-1}), l \right)
		= \left(n \cdot \varphi(l)(n^{-1}), 1 \right)(1, l).$$
		
		Thus the subgroup
		$D := \gen{\gen{n \cdot \varphi(l)(n^{-1})\ |\ l \in L}}_N$
		is contained in $\gen{\gen{L}}_G$.
		It is enough to show that $D \cdot L \triangleleft G$.
		
		$D \triangleleft N$ and $(n, 1)(1, l)(n, 1)^{-1} \in D\cdot L$
		so $N$ normalises $D\cdot L$.
		Also, $L \triangleleft H$ and
		\begin{align*}
			(1, h)\left(n \cdot \varphi(l)(n^{-1}), 1\right)(1, h)^{-1}
			& = \left(\varphi(h)\left(n \cdot \varphi(l)(n^{-1})\right), 1\right) \\
			& = \left(\varphi(h)(n) \cdot \varphi(hl)(n^{-1}), 1\right) \\
			& = \left(\varphi(h)(n) \cdot \varphi(hlh^{-1})\left(\varphi(h)(n^{-1})\right), 1\right) \\
			& = \left(n' \cdot \varphi(l')(n'^{-1}), 1\right)
		\end{align*}
		for $n' := \varphi(h)(n)$,
		which is an element of $N$
		and for $l' := hlh^{-1}$,
		which is an element of $L$ since $L \triangleleft H$.
	\end{proof}
	
	\begin{proposition}\label{normal_sgps}
		$N$ and $K$ are normal subgroups of $G$.
	\end{proposition}
	
	\begin{proof}
		The normality of $N$ is immediate;
		$\varphi(k)(n) = n$ for any $k \in K$ and $n \in N$,
		so by the calculation in \cref{normal_closure_in_semidirect},
		$$\gen{\gen{K}}_G = \gen{\gen{1}}_N \cdot K = K$$
		and so $K$ is normal in $G$.
	\end{proof}

	\subsection{Proof of \cref{characteristic}}\label{ssec:detecting_isos_proof}
	
	The proof is admittedly complicated,
	so we give a plan for it.
	
	\begin{proof}[Overview of the proof of \cref{characteristic}]
		The main steps are as follows.
		\begin{enumerate}
			\item Notice that one of the images
			$\overline{N}$ and $\overline{K}$
			must be trivial.
			\item With the Hopf property of $F_m$ show that
			$\overline N = 1$ implies $M = N$.
			\item Show that $L = H \cap M$ is properly larger than $K$
			using the fact that $T$ isn't free and thus
			$H \twoheadrightarrow T \twoheadrightarrow \overline{H}$
			has kernel larger than $\ker(H \twoheadrightarrow T) = K$.
			\item Deduce that the subgroup $D \triangleleft N$
			in the formula $\gen{\gen{L}}_G = D\cdot L$
			from \cref{normal_closure_in_semidirect} is non-trivial.
			\item Use it to show $C_G(M) < C_N(D)$
			and get a contradiction with assumptions that
			${F_2 \hookrightarrow C_G(M)}$
			and ${F_2 \not \hookrightarrow C_N(D)}$.
		\end{enumerate}
	\end{proof}
	
	\begin{proof}[Proof of \cref{characteristic}]
		We follow the overview and indicate where steps start.
		
		\textbf{Step 1.}
		Consider the images $\overline N$ and $\overline K$
		of $N$ and $K$
		in the quotient $\overline G \cong F_m$,
		which is non-abelian free.
		Since $[N,K]=1$,
		then also in the quotient $[\overline N, \overline K] = 1$.
		
		Now, $\overline N$ and $\overline K$
		are normal in $\overline G \cong F_m$,
		so they are either trivial or non-abelian.
		However, non-abelian subgroups of free groups
		have trivial centralisers.
		Thus $\overline N = 1$, or $\overline K = 1$.
		
		\textbf{Step 2.}
		In the first case the implication is
		that $N < M$,
		but then there would be a surjection
		$$F_m \cong G/N \twoheadrightarrow G/M \cong F_m,$$
		but finitely generated free groups
		are Hopfian, implying that $N = M$,
		which is what we wanted to prove.
		
		\textbf{Step 3.}
		Otherwise $K < M$ and so
		$\overline H$ is a quotient of $H / K = T$.
		Let $$L := \ker(H \twoheadrightarrow \overline H) = H \cap M.$$
		Since $T = H/K$ isn't free, the map
		$$H \twoheadrightarrow T \twoheadrightarrow \overline H < \overline G \cong F_m$$
		must have a non-trivial kernel at the step $T \twoheadrightarrow \overline H$
		and so $L$ is properly larger than $K$.
		
		\textbf{Step 4.}
		$M$ is a normal subgroup of $G$ and $L < M$,
		so $M$ must also contain the normal closure $\gen{\gen{L}}_G$,
		which by \cref{normal_closure_in_semidirect}
		is equal to
		\begin{equation*}
			\gen{\gen{L}}_G =
			\underbrace{\gen{\gen{n \cdot \varphi(l)(n^{-1})\
						|\ l \in L, n \in N}}_N}_{=: D} \cdot L.
		\end{equation*}
		Notably, since $L$ is
		properly larger than $K = 1 \rtimes (\ker \varphi)$
		it contains some element $l_0 \not \in K$
		acting non-trivially,
		i.e. such that there is some $n_0 \in N$
		for which $\varphi(l_0)(n_0) \ne n_0$.
		Thus the subgroup $D$ is non-trivial.
		
		\textbf{Step 5.}
		Finally, since $\gen{\gen{L}}_G < M$,
		then also $C_G(M) < C_G\big(\gen{\gen{L}}_G\big)$, but
		$$C_G\big(\gen{\gen{L}}_G\big) = C_G(DL) = C_G(D) \cap C_G(L).$$
		
		$G$ retracts onto $H$ with kernel $N$,
		$H$ is non-abelian free
		and $L \triangleleft H$ is a non-trivial normal subgroup,
		so $C_H(L) = 1$ and the centraliser $C_G(L)$ is contained in $N$.
		Thus
		$$C_G(M) < C_G(\gen{\gen{L}}_G)
		= C_G(D) \cap C_G(L) = C_N(D) \cap C_N(L) < C_N(D).$$
		$D$ is a non-abelian normal subgroup of $N$
		and $C_N(D)$ can't contain a non-abelian free group,
		and thus $C_G(M)$ also doesn't contain one.
		
		This is a contradiction to the assumption
		that $C_G(M)$ contains an $F_2$,
		so $N = M$.
		
		Now, given two groups $G_i = N \rtimes_{\varphi_i} F_m$ for $i = 1, 2$
		an isomorphism ${\beta: G_1 \to G_2}$ must satisfy $\beta(N) = N$
		since the four conditions characterising $N$ are isomorphism-invariant.
		Now, taking $\alpha = {\beta|_N} \in \aut(N)$
		and $\gamma \in \aut(F_m)$
		induced on $G_1/N \to G_2/N$ we get
		$\alpha\varphi_1\alpha^{-1} = \varphi_2 \circ \gamma$.
	\end{proof}

	\section{Groups with infinitely many T-systems}\label{sec:T-systems}
	
	The previous section and \cref{characteristic} in particular
	showed that the question of isomorphism of the specific groups
	$N \rtimes_{\varphi_1} F_m$ and $N \rtimes_{\varphi_2} F_m$
	which we are constructing
	is equivalent to the existence of
	$\alpha \in \aut(N)$ and $\gamma \in \aut(F_m)$
	such that
	$\alpha\varphi_1(g)\alpha^{-1} = (\varphi_2\circ\gamma)(g)$
	for all $g \in F_m$,
	which we shorthand to $\ad_\alpha\circ\ \varphi_1 = \varphi_2\circ\gamma$.
	We want the groups to be pairwise non-isomorphic,
	so we need to show that $\ad_\alpha\circ\ \varphi_1 \ne \varphi_2\circ\gamma$
	for all $\alpha\in\aut(N)$ and $\gamma\in \aut(F_m)$.
	
	We will, in fact, prove a stronger condition.
	The homomorphisms $\varphi_i$ all have the same image in $\aut(N)$,
	let's call it $T$.
	We will show that for any $\delta \in \aut(T)$ and $\varepsilon \in \aut(F_m)$
	\begin{equation*}
		\delta\circ \varphi_i = \varphi_j \circ \varepsilon\quad\iff\quad i = j.
	\end{equation*}
	
	We start by giving this relation a name.
	
	\begin{definition}
		Two surjective homomorphisms
		$\varphi_i: F_m \twoheadrightarrow T$
		are \emph{Nielsen equivalent}
		if there exists $\varepsilon \in \aut(F_m)$
		such that $\varphi_1 = \varphi_2 \circ \varepsilon$,
		and \emph{T-equivalent}
		if there exist $\delta \in \aut(T)$ and $\varepsilon \in \aut(F_m)$
		such that $\delta \circ \varphi_1 = \varphi_2 \circ \varepsilon$.
	\end{definition}
	The study of Nielsen equivalence and of so-called \emph{T-systems}
	has been an area of active research.
	Notably, though not of direct relevance to us,
	\cite{louder2010nielsen} showed that all of the surjections
	from a finitely generated free group
	to a surface group $\pi_1(S)$ with $\chi(S) \le 0$ are Nielsen equivalent.
	
	\subsection{Torus link groups have infinitely many T-systems}\label{ssec:torus_link_groups}
	
	Most notably for us, there exist rather easy examples of groups
	which have infinitely many non-T-equivalent surjections from $F_2$.
	The examples we give---\emph{torus link groups}---%
	are constructed by amalgamating two copies of $\Z$
	along an infinite subgroup.
	
	An alternative family we could've used is
	hyperbolic $2$-bridge knot groups.
	They have infinitely many Nielsen systems
	of $2$ generators, as shown in \cite{heusener2010generating},
	but on the other hand the fundamental groups
	of their complements in $S^3$
	have finite outer automorphism groups,
	implying that there have infinitely many
	non-equivalent T-systems.
	Moreover, they are virtually special---%
	see \cite[Lemma 17.20 + Corollary 15.3]{wise2021structure}---%
	and techniques similar to those used here can show that they embed
	into $\GL_n(\Z), \out(F_n)$ and $MCG(S_n)$---see \cite{bridson2013subgroups}---%
	for sufficiently large $n$.
	
	\begin{proposition}\label{torus_links_T_systems}
		Let $p, q \in \N$ be such that $p, q \ge 2$ and $p + q \ge 5$.
		and $T$ be the one-relator group given by the presentation
		\begin{equation}\label{torus_link_presentation}
			T(p, q) := \pres{x, y}{x^p=y^q}.
		\end{equation}
		
		Then, there are infinitely many non-T-equivalent surjective maps
		${F_2 \twoheadrightarrow T}$.
	\end{proposition}
	
	\begin{proof}
		\cite{zieschang1977generators} proves that
		any $m$-tuple of generators of $T$
		is Nielsen equivalent to one of the form
		$(x^a, y^b, 1, \ldots, 1)$ 
		where
		\begin{itemize}
			\item $\gcd(a, b)  = \gcd(a, p) = \gcd(b, q) = 1$, and
			\item $0 < 2a \le pb$, and $0 < 2b \le qa$.
		\end{itemize}
		It also proves that any such $m$-tuple generates $T$,
		and that such pairs $(x^a, y^b)$ and $(x^c, y^d)$
		are Nielsen equivalent if and only if
		$a = c$ and $b = d$.
		
		Additionally, \cite{schreier1924uber} showed
		that an automorphism of $T$ sends
		\begin{equation*}
			x \mapsto t\ \sigma(x)^\varepsilon t^{-1},
			\quad y \mapsto t\ \sigma(y)^\varepsilon t^{-1},
		\end{equation*}
		for $\varepsilon = \pm 1$, $t \in T$
		and $\sigma$ being either identity,
		or possibly the swap $x \leftrightarrow y$,
		but only if $a = b$.
		None of these maps change the Nielsen class of a pair,
		so this implies that
		$T$ has indeed infinitely many
		non-T-equivalent generating pairs.
	\end{proof}
	
	\subsection{Finding torus link groups in $\out(N)$}\label{ssec:finding_torus_link_gps}
	
	In order to be able to construct our examples
	of infinite families of extensions $N$-by-$F_2$
	which are non-isomorphic,
	we will need to find a torus link group $T=T(p,q)$
	with $p, q \ge 2$ and $p + q \ge 5$
	in the outer automorphism group $\out(N)$.
	
	In all of the cases we consider
	the embedding $T \hookrightarrow \out(N)$
	factors through the map $\aut(N) \twoheadrightarrow \out(N)$
	
	\begin{proposition}\label{T_is_free_by_cyclic}
		Let $T = T(p,q)$ be the torus link group
		defined by the presentation~\ref{torus_link_presentation}.
		Let $\pi: T \to \left(\Z/\lcm(p,q)\right)$ be defined by
		$\pi(x) = \lcm(p,q) /p$ and $\pi(y) = \lcm(p, q) / q$.
		
		Then $\ker \pi \cong F_k \times \Z$
		where ${k = (pq - p - q)/\gcd(p,q) + 1}$.
	\end{proposition}
	
	\begin{proof}
		The centre of $T$ is the subgroup $Z = \gen{x^p} = \gen{y^q}$,
		which lies in the kernel of $\pi$,
		so $\pi$ factors through $G \twoheadrightarrow G/Z$,
		which is isomorphic to
		$\left(\Z/ p\right) * \left(\Z / q\right)$.
		
		All torsion elements of $G/Z$ inject under $G/Z \twoheadrightarrow \left(\Z / \lcm(p, q)\right)$,
		so $(\ker \pi) / Z$ is a torsion-free virtually free group
		and thus it must be free.
		We can compute its Euler characteristic
		\begin{equation*}
			\chi\left((\ker \pi) / Z\right)
			= \lcm(p, q) \cdot \chi\left((\Z/p) * (\Z / q)\right)
			= \lcm(p, q) \cdot \left(1/p + 1/q - 1\right).
		\end{equation*}
		Since $pq = \lcm(p, q) \cdot \gcd(p, q)$
		for $k = (pq - p - q)/\gcd(p,q) + 1$
		we get $(\ker \pi)/Z \cong F_k$.
		
		Finally, $\ker \pi$ is a central extension
		of $\Z$ by $F_k$ and any such extension
		is isomorphic to $F_k \times \Z$.
	\end{proof}
	
	As a result, we got that $T = T(3,3)$
	is a $(F_2 \times \Z)$-by-$(\Z/3)$ extension.
	
	For embedding $T$ into outer automorphism groups
	we will need the Universal Embedding Theorem,
	likely proved first in \cite{krasner1951produit}.
	
	\begin{lemma}[Universal Embedding Theorem]\label{universal_embedding}
		Let $N$ and $Q$ be groups,
		and $G$ be an $N$-by-$Q$ extension.
		Then $G$ embeds into the wreath product $N \wr Q$.
	\end{lemma}
	
	Thus, in order to ensure an embedding $T \hookrightarrow \out(N)$
	it is enough to find $T' := (F_2 \times \Z) \wr (\Z/ 3)$,
	which is
	$$\left((F_2 \times \Z)\times (F_2 \times \Z)\times (F_2 \times \Z)\right) \rtimes (\Z/3)$$
	with $(\Z/3)$ permuting the factors cyclically.
	
	\begin{proposition}\label{embeds_into_Out(Fn)}
		The group $T' := (F_2 \times \Z) \wr (\Z/ 3)$
		embeds into $\aut(F_9)$
		and thus also into $\out(F_n)$
		for any $n \ge 10$.
	\end{proposition}
	
	\begin{proof}
		We need to find three copies of $F_2 \times \Z$
		commuting with each other
		and an automorphism $\sigma$ which permutes them.
		
		Let's fix a free basis $x_0, \ldots, x_8$
		and define homomorphisms
		$f_i, g_i, h_i$ for ${i = 0, 1, 2, }$
		and $\sigma$ as follows.
		\begin{itemize}
			\item $f_i$ fixes all basis elements except $x_{3i}$
			and sends $x_{3i} \mapsto x_{3i}\cdot x_{3i+1}$,
			\item $g_i$ fixes all basis elements except $x_{3i}$
			and sends $x_{3i} \mapsto x_{3i}\cdot x_{3i+2}$,
			\item $h_i$ fixes all basis elements except $x_{3i}$
			and sends $x_{3i} \mapsto x_{3i+1} \cdot x_{3i}$,
			\item $\sigma(x_i) := x_{i + 3}$, for $i = 0, \ldots, 8$,
		\end{itemize}
		where the indices are taken modulo $9$.
		Then $\gen{f_i, g_i} \cong F_2$, $\gen{h_i} \cong \Z$,
		these factors all commute and $\sigma$ permutes them as needed.
		
		Finally, for $n \ge 10$ we have $F_n \cong F_9 * F_{n-9}$
		and we can extend any automorphism of $F_9$
		by identity on the factor $F_{n-9}$.
		Since $F_{n-9}$ is non-trivial,
		none of these extended automorphisms is inner,
		so $T'$ indeed embeds into $\out(F_n)$.
	\end{proof}
	
	\begin{proposition}\label{embeds_into_Out(Sigma)}
		For $n \ge 5$,
		the group $T' := (F_2 \times \Z) \wr (\Z/ 3)$
		embeds into $\mcg(S_n)$---%
		the mapping class group
		of the closed orientable surface of genus $n$%
		---which is an index $2$ subgroup of $\out(\Sigma_n)$,
		where $\Sigma_n$ is the fundamental group
		of a genus $n$ closed orientable surface.
		Furthermore, the embedding factors through
		$\aut(\Sigma_n) \twoheadrightarrow \out(\Sigma_n)$.
	\end{proposition}
	
	\begin{proof}
		\cref{order3} shows decompositions
		with a symmetry of order $3$
		which extends to a homeomorphism $R$
		of the surface $S := S_n$.
		We give simple closed curves
		$\alpha_i, \beta_i, \gamma_i \sub S$
		for $i = 0, 1, 2$ such that
		$$R\colon\quad
		\alpha_i \mapsto \alpha_{i+1},\quad
		\beta_i \mapsto \beta_{i+1},\quad
		\gamma_i \mapsto \gamma_{i+1}$$
		and such that all all are disjoint
		apart from $\alpha_i$ and $\beta_i$ for $i = 0, 1, 2$,
		which have one intersection.
		
		This implies---according to the classification of \cite{ishida1996structure} of
		subgroups of $\mcg(S)$ generated by two Dehn twists---%
		that
		$\gen{T_{\alpha_i}^2, T_{\beta_i}} \cong F_2$
		and $\gen{T_{\gamma_i}} \cong \Z$,
		and that these all groups commute with each other,
		while being permuted in a cyclic manner by $R$.
		
		This gives a copy of $T'$ inside $\mcg(S_n)$,
		which is isomorphic to
		an index $2$ subgroup of $\out(\Sigma_n)$
		by Dehn-Nielsen-Baer Theorem---%
		see \cite[Theorem 8.1.]{farb2011primer}.
		
		Since $R$ and the $T_{\alpha_i}, T_{\beta_i}, T_{\gamma_i}$
		for $i = 0, 1$ and $2$ fix points of the surface
		we can set the basepoint to be one of them
		and factor the embedding through
		$\aut(\pi_1(S_n)) \twoheadrightarrow \out(\pi_1(S_n))$.
		
		\begin{figure}[p]
			\begin{center}
				\begin{tikzpicture}[l/.style = {font=\boldmath},]
					
					\coordinate (offset0) at (0, 0.35);
					\coordinate (offset1) at (0, 1);
					\coordinate (offset2) at (0, 0.05);
					
					\coordinate (gamma0) at (-1.4, 0);
					\coordinate (alpha0) at (-2.2, 0.1);
					\coordinate (beta0) at (-2.74, -0.72);
					\coordinate (sigma) at (0.4, 1.4);
					
					\node[inner sep=0pt] (mod0) at (0,11.5)
					{\includegraphics[page=2, height=8cm, trim={3cm 16cm 3cm 1cm}]{pictures.pdf}};
					\node[inner sep=0pt] (mod1) at (0,5)
					{\includegraphics[page=3, height=8cm, trim={3cm 16cm 3cm 1cm}]{pictures.pdf}};
					\node[inner sep=0pt] (mod2) at (0,0)
					{\includegraphics[page=1, height=8cm, trim={3cm 16cm 3cm 1cm}]{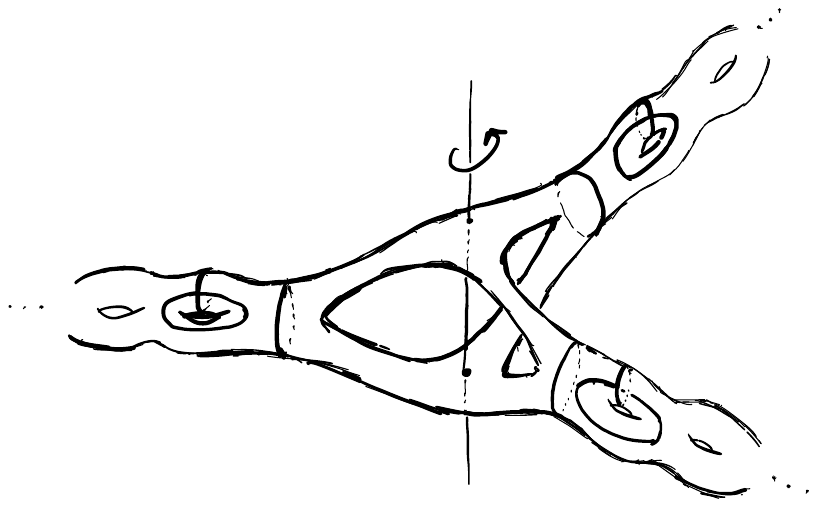}};
					
					\node[l] at ($(mod0)+(gamma0)+(offset0)$) {$\gamma_0$};
					\node[l] at ($(mod1)+(gamma0)+(offset1)$) {$\gamma_0$};
					\node[l] at ($(mod2)+(gamma0)+(offset2)$) {$\gamma_0$};
					
					\node[l] at ($(mod0)+(alpha0)+(offset0)$) {$\alpha_0$};
					\node[l] at ($(mod1)+(alpha0)+(offset1)$) {$\alpha_0$};
					\node[l] at ($(mod2)+(alpha0)+(offset2)$) {$\alpha_0$};
					
					\node[l] at ($(mod0)+(beta0)+(offset0)$) {$\beta_0$};
					\node[l] at ($(mod1)+(beta0)+(offset1)$) {$\beta_0$};
					\node[l] at ($(mod2)+(beta0)+(offset2)$) {$\beta_0$};
					
					\node[l] at ($(mod0)+(sigma)+(offset0)$) {$R$};
					\node[l] at ($(mod1)+(sigma)+(offset1)$) {$R$};
					\node[l] at ($(mod2)+(sigma)+(offset2)$) {$R$};
				\end{tikzpicture}
			\end{center}
			\caption{Decomposition of $S_g$ depending on $g \mod 3$.}
			\label{order3}
		\end{figure}
		
	\end{proof}
	
	\begin{proposition}\label{embeds_into_GLnZ}
		For $n \ge 12$,
		the group $T' := (F_2 \times \Z) \wr (\Z/ 3)$
		embeds into $\GL_n(\Z) \cong \out(\Z^n)$.
	\end{proposition}
	
	\begin{proof}
		The matrices
		$$A = \begin{pmatrix}
			1 & 2 \\
			0 & 1
		\end{pmatrix}, \quad
		B = \begin{pmatrix}
			1 & 0 \\
			2 & 1
		\end{pmatrix}$$
		generate a subgroup of $\GL_2(\Z)$
		isomorphic to $F_2$.
		Thus, the subgroup generated by
		$12\times 12$ block-diagonal matrices
		\begin{align*}
			a_0 = \diag(A, I, I, I, I, I),\
			a_1 = \diag(I, I, A, I, I, I),\
			a_2 = \diag(I, I, I, I, A, I), \\
			b_0 = \diag(B, I, I, I, I, I),\ 
			b_1 = \diag(I, I, B, I, I, I),\
			b_2 = \diag(I, I, I, I, B, I), \\
			c_0 = \diag(I, A, I, I, I, I),\
			c_1 = \diag(I, I, I, A, I, I),\
			c_2 = \diag(I, I, I, I, I, A)
		\end{align*}
		is actually isomorphic to
		$(F_2 \times \Z)\times (F_2 \times \Z)\times (F_2 \times \Z)$
		and if we additionally include the matrix $\sigma$,
		which permutes the blocks by 2 to the right,
		we get a subgroup isomorphic to $T'$.
	\end{proof}

	\section{Final results}\label{sec:final_results}
	
	In this section we combine the results proven in previous sections
	to prove \Cref{mainthm:general} -- a general tool
	for producing families of non-isomorphic semidirect products
	$N \rtimes F_m$ which share profinite completions.
	
	This, together with the constructions done in \cref{sec:T-systems},
	allows us to produce in \cref{mainthm:infinite_genus_families}
	three infinite families of non-isomorphic groups of type $F$
	sharing the same profinite completion.
	One should note that the assumptions
	that $n \ge 10$ for $N = F_n$,
	$n \ge 5$ for $N = \Sigma_n$,
	and $n \ge 12$ for $N = \Z^n$
	are chosen for convenience only
	and aren't implied to be minimal
	for which the phenomenon occurs.
	
	\addtocounter{thmx}{-2}
	\begin{thmx}\label{mainthm:general}
		Let $N$ be a
		finitely generated,
		residually finite group
		such that
		the centralisers $C_N(M)$
		of non-trivial normal subgroups $M \triangleleft N$
		don't contain a non-abelian free group.
		
		Let $T < \aut(N)$ be a subgroup which isn't free.
		Let $m \ge 2$ and $\{\varphi_i\colon F_m \twoheadrightarrow T\}$
		be a family of surjective homomorphisms such that,
		setting $\overline{\varphi_i}$ to be the composition of $\varphi_i$
		with ${\nu\colon \aut(N) \twoheadrightarrow \out(N)}$
		and $\overline T = \nu(T)$,
		we have
		\begin{equation*}
			\delta \circ \overline{\varphi_i} = \overline{\varphi_j} \circ \varepsilon\
			\text{for some $\delta \in \aut(\overline T)$,
				$\varepsilon \in \aut(F_m)$}\
			\iff\ i = j.
		\end{equation*}
		
		Then the groups $G_i := N \rtimes_{\varphi_i} F_m$
		are residually-finite
		and pairwise non-isomorphic,
		while having isomorphic profinite completions.
	\end{thmx}
	
	\begin{proof}
		Group $N$ together with homomorphisms
		$\varphi_i\colon F_m \to T < \aut(N)$
		satisfy the assumptions of \cref{characteristic}
		and thus the family $\cD$ of extensions
		$$
		\big(
		\begin{tikzcd}[column sep=1.5em]
			1 \arrow{r}	&
			\underbrace{N \rtimes 1 \arrow{r}{\iota_i}}_N &
			\underbrace{N \rtimes_{\varphi_i} F_m}_{G_i} \arrow{r}{\pi_i}
			& F_m \arrow{r} & 1
		\end{tikzcd}
		\big)
		$$
		detects isomorphisms.
		\Cref{different_maps_to_out} then
		shows that the groups $G_i$
		are pairwise non-isomorphic.
		
		On the other hand,
		the fact that $\varphi_i$ all surject
		the same group $T < \aut(N)$
		proves that $\widehat{G_i}$
		are all isomorphic,
		as shown in \cref{semidirect_products_profinitely_isomorphic}.
	\end{proof}
	
	\begin{thmx}\label{mainthm:infinite_genus_families}
		Let $N$ be one of the following.
		\begin{itemize}
			\item Free group $F_n$ with $n \ge 10$.
			\item Fundamental group $\Sigma_n$
			of a closed orientable surface with genus $n \ge 5$, or
			\item Free abelian group $\Z^n$ with $n \ge 12$.
		\end{itemize}
		Then there exists an infinite family
		of non-isomorphic groups
		$G_i := N \rtimes_{\varphi_i} F_2$
		having isomorphic profinite completions.
		These groups are of type $F$.
	\end{thmx}
	
	\begin{proof}
		The groups $N$ satisfying the assumptions
		are finitely generated and residually finite;
		their centralisers don't contain non-abelian free groups.
		
		In \cref{sec:T-systems} we proved that
		the torus link group $T = T(3, 3)$---%
		which isn't free---%
		embeds into $\out(N)$
		in all of the cases covered by the assumptions;
		furthermore, the embedding factors through
		$\aut(N) \twoheadrightarrow \out(N)$.
		
		Since $T$ has infinitely many non-T-equivalent surjections
		$\varphi_i: F_2 \twoheadrightarrow T$,
		all of the assumptions of \cref{mainthm:general}
		are satisfied,
		and so the groups $N \rtimes_{\varphi_i} F_2$
		form an infinite family of non-isomorphic groups
		which share the profinite completion.
		
		\cite{wall1961resolutions} shows
		that given finite resolutions of $\Z$
		by finitely-generated free
		$\Z N$- and $\Z Q$-modules respectively,
		we can construct a finite resolution of $\Z$
		by finitely-generated free $\Z G$-modules.
		This shows in particular that
		any extension of $N$ by $Q$ is of type $FL$.
		Finally, a finitely presented group of type $FL$
		is also of type $F$, by \cite[VIII.7.1]{brown2012cohomology}.
	\end{proof}
	
	%%%%%%%%%%%%%%%%%%%%%% REFERENCES HERE
	
	\bibliographystyle{siam}
	\bibliography{Bibliography}
	
\end{document}